\documentclass[11pt]{amsart}
\usepackage{amsmath}
\usepackage{amssymb}
\usepackage{amsthm}
\usepackage{mathrsfs}
\usepackage[all,cmtip]{xy}
\usepackage{times}

\usepackage[OT2,T1]{fontenc}
\DeclareSymbolFont{cyrletters}{OT2}{wncyr}{m}{n}
\DeclareMathSymbol{\Sha}{\mathalpha}{cyrletters}{"58}

\usepackage{enumerate}
\def\C{{\mathcal C}}

\def\N{{\mathbb N}}
\def\G{{\mathscr G}}

\def\K{{\mathscr K}}

\def\T{{\mathcal T}}

\def\S{{\mathscr S}}
\def\ra{\rightarrow}
\def\ora{\stackrel{\leq}{\longrightarrow}}

\def\natl{\Longrightarrow}

\def\ep{\varepsilon}

\def\<{\langle}
\def\>{\rangle}
\def\leq{\leqslant}
\def\geq{\geqslant}

\def\obj{\operatorname{obj}}

\def\lam{\lambda}

\def\id{\operatorname{id}}

\def\endomo{\operatorname{End}}
\def\aut{\operatorname{Aut}}
\def\hol{\operatorname{Hol}}
\def\prem{\operatorname{Prem}}
\def\aff{\operatorname{Aff}}

\def\gpd{\textbf{Gpd}}
\def\ogpd{{\sf OGPD}}
\def\endo{{\sf END}}
\def\homogpd{\textbf{OGpd}}

\def\ol{\overline}

\def\dom{\mathbf d}
\def\ran{\mathbf r}

\numberwithin{equation}{section}

\title{Ordered groupoids and the holomorph of an inverse semigroup}

\author{N.D. Gilbert and E.A. McDougall}
\address{
School of Mathematical and Computer Sciences\\
and the Maxwell Institute for the Mathematical Sciences,\\
Heriot-Watt University, Edinburgh EH14 4AS, U.K.}
\email{N.D.Gilbert@hw.ac.uk, eam16@hw.ac.uk}

\thanks{The second author gratefully acknowledges the support of a Summer Vacation Scholarship from the Carnegie Trust for the Universities 
of Scotland, which made this collaboration possible.}

\date{}
\newtheorem{theorem}{Theorem}[section]
\newtheorem{prop}[theorem]{Proposition}
\newtheorem{lemma}[theorem]{Lemma}
\newtheorem{cor}[theorem]{Corollary}

\theoremstyle{definition}

\begin{document}
\thispagestyle{empty}

\begin{abstract}
We present a construction for the holomorph of an inverse semigroup, derived from the cartesian closed structure of the category of ordered
groupoids.  We compare the holomorph with the monoid of mappings that preserve the ternary heap operation on an inverse semigroup: 
for groups these two constructions coincide.  We present detailed calculations for semilattices of groups and for the polycyclic monoids.
\end{abstract}




\maketitle

\section*{Introduction}
The \emph{holomorph} of a group $G$ is the semidirect product of $\hol(G) = \aut(G) \ltimes G$ of $G$ and its
automorphism group (with the natural action of $\aut(G)$ on $G$).  The embedding of $\aut(G)$ into the symmetric
group $\Sigma_G$ on $G$ extends to an embedding of $\hol(G)$ into $\Sigma_G$ where $g \in G$ is identified with
its (right) Cayley representation $\rho_g : a \mapsto ag$.  Then $\hol(G)$ is the normalizer of $G$ in $\Sigma_G$
(\cite[Theorem 9.17]{Ro}).  A second interesting characterization of $\hol(G)$ is due to Baer \cite{Baer} (see also 
\cite{cert} and \cite[Exercise {\bf 520}]{Ro}).  The \emph{heap} operation on $G$ is the ternary operation defined by
$\< a,b,c \> = ab^{-1}c$.  Baer shows that a subset of $G$ is closed under the heap operation if and only if it is
a coset of some subgroup, and that the subset of $\Sigma_G$ that preserves $\< \cdots \>$ is precisely $\hol(G)$:
that is, if $\sigma \in \Sigma_G$, then for all $a,b,c, \in G$ we have
\[ \< a,b,c \> \sigma = \< a \sigma, b \sigma, c \sigma \> \]
if and only if $\sigma \in \hol(G)$.

The holomorph also arises naturally from category-theoretic considerations.  The category of groups embeds in the
category $\gpd$ of groupoids, which is cartesian closed. We therefore have a bifunctor ${\sf GPD}$ that associates to any two groupoids 
$A,B$ a groupoid ${\sf GPD}(A,B)$, whose objects are the functors $A \ra B$ and whose arrows are natural transformations
between functors.  It follows that to any group $G$ we may associate the groupoid ${\sf GPD}(G,G)$ and this will be an internal
monoid in the category of groupoids, whose objects are the endomorphisms $G \ra G$.  The full subgroupoid on the automorphism
group $\aut(G)$ is then an internal group in groupoids, and its group structure is precisely the holomorph $\hol(G)$.

Our aim in this paper is to produce a candidate for the holomorph of an inverse semigroup $S$.  Because of the close
connections between inverse semigroups and ordered groupoids we follow the category-theoretic approach that we have just
outlined, embedding the category of inverse semigroups into the category $\homogpd$ of ordered groupoids and so obtaining from an inverse
semigroup $S$ an ordered groupoid $\vec{S}$, and using the cartesian
closed structure there to produce an internal monoid $\ogpd(\vec{S},\vec{S})$ in the category of ordered groupoids. The objects of 
$\ogpd(\vec{S},\vec{S})$ are the ordered functors $\vec{S} \ra \vec{S}$ and the arrows are natural transformations.  
We identify $\hol(S)$ as the monoid $\ogpd(\vec{S},\vec{S})$.  This is a semidirect product of the monoid of \emph{premorphisms}
of $S$ (introduced by McAlister \cite{McA1} as $v$--prehomomorphisms) and a monoid of ordered functions on the semilattice of 
idempotents $E(S)$, related to the \emph{flow monoid} of \cite{ndgflow}.  We compare $\hol(S)$ with the collection $\Sha(S)$ of
functions that preserve the heap operation on $S$, and discuss in detail the cases when $S$ is a semilattice of groups and a
polycyclic monoid.

\section{Premorphisms}
Let $S$ be an inverse semigroup.
We denote by $E(S)$ the set of all idempotents of $S$.
Recall that the \emph{natural partial order}  $\leq$ on $S$ is
defined by
\[
s\leq t \Leftrightarrow (\exists e\in E(S))(s=te) \,. \]
It is well known that $(E(S),\leq)$ forms a semilattice.

\begin{lemma}
\label{npo}
Let $S$ be an inverse semigroup and $a,b \in S$.  Then the following are equivalent:
\begin{enumerate}[(i)]
\item $a \leq b$,
\item there exists $f \in E(S)$ such that $a=bf$,
\item $a = aa^{-1}b$,
\item $a=ba^{-1}a$.
\end{enumerate}
\end{lemma}

\begin{lemma}
\label{archy}
If $S$ is an inverse semigroup and $x \in S$ satisfies $x \leq x^2$ then $x=x^2$: that is, $x$ is an idempotent.
\end{lemma}

\begin{proof}
By Lemma \ref{npo}, $x \leq x^2$ implies that $x = xx^{-1}x^2$: but $xx^{-1}x^2 = (xx^{-1}x)x=x^2$.
\end{proof}

Let $S$ and $T$ be inverse semigroups.  A function $\theta: S \ra T$ is a  {\em premorphism} if, for all $a,b \in S$,
$(ab) \theta \leq a \theta b \theta$.  Premorphisms were introduced by McAlister, under the name $v$--prehomomorphisms,
in \cite{McA1}.  We collect some useful facts about premorphisms from in the next two results.

\begin{lemma}
\label{premorph1}
Let $\theta : S \ra T$ be a premorphism.  Then:
\begin{enumerate}[(a)]
\item if $e \in E(S)$ then $e \theta \in E(T)$,
\item for all $a \in S$ we have $a^{-1}\theta = (a \theta)^{-1}$.
\end{enumerate}
\end{lemma}

\begin{proof}
(a) For $e \in E(S)$, $e \theta = e^2 \theta \leq e \theta e \theta$ and so by Lemma \ref{archy} we have $e \theta \in E(T)$.

(b) Since $a = aa^{-1}a$ we have $a \theta \leq a \theta a^{-1}\theta a \theta$ and so $a \theta a^{-1}\theta \leq a \theta a^{-1}\theta a \theta a^{-1} \theta$.
Again by Lemma \ref{archy}, $a \theta a^{-1} \theta$ is an idempotent, and so $a \theta a^{-1}\theta = a \theta a^{-1}\theta a \theta a^{-1} \theta$.
Multiplying in the right by $a \theta$, we deduce that 
\[ a \theta = a \theta a^{-1}\theta a \theta =  a \theta a^{-1}\theta a \theta .\]  
Similarly
\[ a^{-1} \theta = a^{-1} \theta a\theta a^{-1} \theta =  a^{-1} \theta a\theta a^{-1} \theta \]  
and hence $a^{-1} \theta = (a \theta)^{-1}$.
\end{proof}

\begin{prop} 
\label{premeq}
Let $S$ and $T$ be inverse semigroups.  A function $\theta: S \ra T$ is a premorphism if and only if
\begin{itemize}
\item $\theta$ is ordered,
\item if $a^{-1}a=bb^{-1}$ then $(ab)\theta = a \theta b \theta$.
\end{itemize}
\end{prop}

\begin{proof}
Suppose that $\theta$ has the two properties stated in the Proposition.  Set $x=abb^{-1}$ and $y=a^{-1}ab$.  Then $ab=xy$ with $x \leq a, y \leq b$, and
\[ x^{-1}x = bb^{-1}a^{-1}abb^{-1} = a^{-1}abb^{-1} = yy^{-1} .\] 
Hence $(ab)\theta = (xy)\theta = x \theta y \theta \leq a \theta b \theta$ and $\theta$ is a premorphism.

Conversely, suppose that $\theta$ is a premorphism.  If $a,s \in S$ with $a \leq s$ then $a=es$ for some $e \in E(S)$ and so
\[ a \theta = (es)\theta \leq (e \theta)(s \theta) \leq s \theta \]
since $e \theta \in E(T)$ by Lemma \ref{premorph1}(a).  Hence $\theta$ is ordered.  Now if $a,b \in S$ with $a^{-1}a=bb^{-1}$ we have
\begin{align*}
a \theta b \theta &= a \theta (bb^{-1}b)\theta = a \theta (a^{-1}ab) \theta \\
& \leq a \theta a^{-1} \theta (ab)\theta \\
&= a \theta (a \theta)^{-1} (ab) \theta \leq (ab) \theta \,.
\end{align*}
\end{proof}

The property that $a^{-1} \theta = (a \theta)^{-1}$ was included in the original definition of a premorphism in \cite{McA1}: its redundancy was noted in
\cite{McA2}.  Proposition \ref{premeq} is stated as part of Theorem 3.1.5 in \cite{mvlbook}.  

\begin{cor}
\label{premorph2}
If $\theta : S \ra T$ is a premorphism then for all $s \in S$, $(ss^{-1})\theta = s \theta \,(s^{-1})\theta = s \theta \,(s \theta)^{-1}$.
\end{cor}

\begin{proof}
This follows from Lemma \ref{premorph1}(b) and Proposition \ref{premeq}.
\end{proof}

We record from \cite{McA1} the following generalisation of part of 
Proposition \ref{premeq}: the proof of that result is easily adapted.

\begin{prop}{\cite[Lemma 1.4]{McA1}}
\label{prem_is_hom}
Let $\theta : S \ra T$ be a premorphism and suppose that $a,b \in S$ satisfy either that $a^{-1}a \geq bb^{-1}$ or that $a^{-1}a \leq bb^{-1}$.  Then
$(ab)\theta = a \theta b \theta$.
\end{prop}

The set of all premorphisms $S \ra T$ is denoted by $\prem(S,T)$: we write $\prem(S)$ for $\prem(S,S)$.  It is clear that the composition of two premorphisms is a
premorphism, and so $\prem(S)$ is a monoid.

\section{Inverse semigroups and ordered groupoids}
\label{is_and_og}
A {\em groupoid} $G$ is a small category in which every morphism
is invertible.  We consider a groupoid as an algebraic structure
following \cite{pjhbook}: the elements are the morphisms, and
composition is an associative partial binary operation.
The set of identities in $G$ is denoted $E(G)$, and an element $g \in G$ has domain
$g\dom=gg^{-1}$ and range $g\ran=g^{-1}g$.    For each $x \in E(G)$ the set
$G(x) = \{ g \in G : g \dom = x = g \ran \}$ is a subgroup of $G$, called the \emph{local subgroup} at $x$.
A groupoid $G$ is connected if, for any $x,y \in E(G)$ there exists $g \in G$ with $g \dom =x$ and $g \ran =y$.
In a connected groupoid, all local subgroups are isomorphic, and for any such local subgroup $L$ there is an
isomorphism $E(G) \times L \times E(G)$, where the latter set carries the groupoid composition
$(x,k,y)(y,l,z) = (x,kl,z)$.  

An {\em ordered groupoid} $(G,\leq)$ is a groupoid $G$
with a partial order $\leq$ satisfying the following axioms:
\begin{enumerate}
\item[OG1] for all $g,h \in G$, if $g \leq h$ then $g^{-1} \leq h^{-1}$,
\item[OG2] if $g_1 \leq g_2 \,, h_1 \leq h_2$ and if the compositions
$g_1h_1$ and $g_2h_2$ are   defined, then $g_1h_1 \leq g_2h_2$,
\item[OG3] if $g \in G$ and $x$ is an identity of $G$ with $x \leq
g \dom$, there exists a unique element $(x|  g)$, called the {\em restriction}
of $g$ to $x$, such that $(x|g)\dom=x$ and $(x|g) \leq g$,
\end{enumerate}
As a consequence of [OG3] we also have:
\begin{enumerate}
\item[OG3*] if $g \in G$ and $y$ is an identity of $G$ with $y \leq
g \ran$, there exists a unique element $(g|  y)$, called the {\em corestriction}
of $g$ to $y$, such that $(g|y)\ran=y$ and $(g|  y) \leq g$,
\end{enumerate}
since the corestriction of $g$ to $y$ may be defined as $(y |  g^{-1})^{-1}$.

Let $G$ be an ordered groupoid and let $a,b \in G$.  If 
$a \ran$ and $b \dom$ have a greatest lower bound $\ell \in E(G)$, then we may define the
{\em pseudoproduct} of $a$ and $b$ in $G$ as
$a \otimes b = (a | \ell)(\ell|  b)$,
where the right-hand side is now a composition defined in $G$. As 
Lawson shows in Lemma 4.1.6 of \cite{mvlbook}, this is a partially
defined associative operation on $G$. 

If $E(G)$ is a meet semilattice then $G$ is called an \emph{inductive} groupoid. The pseudoproduct  is then everywhere defined and
$(G,\otimes)$ is an inverse semigroup.  On the other hand, given an inverse semigroup $S$ with
semilattice of idempotents $E(S)$, then $S$ is a poset under the natural partial order, and the
restriction of its multiplication to the partial composition
\[ a \cdot b = ab \in S \; \text{defined when} \; a^{-1}a=bb^{-1} \]
gives $S$ the structure of an ordered groupoid, which we denote by $\vec{S}$.  These constructions 
give an isomorphism between the categories of inverse semigroups and inductive groupoids: this is the \emph{Ehresmann-Schein-Nambooripad
Theorem} \cite[Theorem 4.1.8]{mvlbook}.

Proposition \ref{premeq} above records the details of the correspondence between morphisms in the Ehresmann-Schein-Nambooripad
Theorem: ordered functors between inductive groupoids correspond to premorphisms of inverse semigroups.

\section{The category of ordered groupoids is cartesian closed}
\label{ogpdiscc}
We can now use constructions for ordered groupoids to derive constructions for inverse semigroups, and the key construction for this paper
will be the cartesian closed structure on the category $\homogpd$ of ordered groupoids.  This gives, for any two ordered groupoids
$A,B$ an \emph{internal hom functor} $\ogpd(A,B)$ that is again an ordered groupoid.  If $A,B$ are inductive then
$\ogpd(A,B)$ need not be inductive, and so to obtain a construction applicable to inverse semigroups we need to use the
larger category of ordered groupoids.  This is analogous to the construction of the holomorph of a group via the 
internal hom functor on the category of groupoids described in the introduction.

The cartesian closed structure on $\homogpd$ is just the ordered version of the well-known cartesian closed structure on 
$\gpd$, but we give a detailed account of it here to clarify the later application to inverse semigroups.  An informative and more detailed discussion,
including further applications of these ideas,may be found in \cite[Appendix C]{BrHiSi}.

Let $A,B$ be ordered groupoids.
The objects of $\ogpd(A,B)$ are the ordered functors $A \ra B$.  Given two such ordered functors $f,g : A \ra B$,
an arrow in $\ogpd(A,B)$ from $f$ to $g$ is an ordered natural transformation $\tau$ from $f$ to $g$: that is,
$\tau$ is an ordered function $\obj(A) \ra B$ such that, for each arrow $a \in A$ with $a \dom = x$ and $a \ran =y$,
the square
\[ \xymatrixcolsep{3pc}
\xymatrix{
xf \ar[d]_{x \tau} \ar[r]^{af} & yf \ar[d]^{y \tau}\\
xg \ar[r]_{ag} & yg}
\]
in $B$ commutes.  We write $\tau: f \natl g$.  Note that for all $x \in \obj(A)$ we have $(x \tau)\dom =(x \tau)(x \tau)^{-1}=xf$.  Now $f$ and $\tau$ determine $g$, since for any $a \in A$ we have
$ag = ((a \dom) \tau)^{-1} (af) ((a \ran)\tau)$.  Given ordered natural transformations $\tau: f \natl g$ and $\sigma: g \natl h$
their composition is the ordered natural transformation $\tau \cdot \sigma : f \natl h$ defined by
$x(\tau \cdot \sigma) = (x \tau)(x \sigma)$.  (Note that $(x \sigma)\dom = xg = (x \tau)\ran$.) This makes $\ogpd(A,B)$ a groupoid, since an ordered natural transformation $\tau$ has inverse $\ol{\tau} : x \mapsto (x \tau)^{-1}$.

If $p,x \in \obj(A)$ and $p \leq x$ then $pf \leq xf$ and $p \tau \leq x \tau$ with $(p \tau)\dom = pf$.  Hence we have
$p \tau = (pf| x \tau)$ and, if every object of $A$ is below a maximal object, then $\tau$ is determined by its values on the maximal 
objects of $\obj(A)$.  In the special case that $\obj(A)$ has a maximum $m$, then $\tau$ is determined by $m \tau$ and for
all $x \in \obj(A)$ we have $x \tau = (xf | m \tau)$.

\begin{lemma}
If $A,B$ are ordered groupoids then $\ogpd(A,B)$ is also an ordered groupoid.
\end{lemma}

\begin{proof}
We have already described the underlying groupoid structure.  For the ordering on $\ogpd(A,B)$, suppose that
$f,g : A \ra B$ are ordered functors and that $f \leq g$: that is, for all $a \in A$ we have $af \leq ag$.
Suppose that $\sigma: g \natl h$, so that for all $x \in \obj(A)$ we have $(x \sigma)\dom = xg$.  Then $xf \leq xg$
and so $x \sigma$ has a unique restriction $(xf|x \sigma)$ to $xf$ in $B$.  The restriction of $\sigma$ to $f$ is then
defined by $x(f|\sigma) = (xf | x \sigma)$.  This is an ordered function $\obj(A) \ra B$ and defines an ordered
natural transformation from $f$.  Moreover, suppose that $\tau: f \natl k$ and that $\tau \leq \sigma$.  Then
for all $x \in \obj(A)$, we have $(x \tau)\dom = xf$ and $x \tau \leq x \sigma$.  Hence $x \tau = (xf|x \sigma)$
and so $\tau = (f| \sigma)$.
\end{proof}

We shall now identify an arrow in $\ogpd(A,B)$ with a pair $(f,\tau)$ where $f: A \ra B$ is an ordered functor and
$\tau: f \natl g$ is an ordered natural transformation. As already remarked, $f$ and $\tau$ determine $g$.
We now have an ordered functor $\ep:A \times \ogpd(A,B) \ra B$ given by 
\[ \ep : (a,(f,\tau)) \mapsto (af)(a \ran)\tau. \]

\begin{lemma}
Given ordered groupoids $A,B$ and $C$ and an ordered functor $\gamma : A \times B \ra C$ there exists a unique ordered functor
$\lam : B \ra \ogpd(A,C)$ such that the diagram
\[ \xymatrixcolsep{5pc}
\xymatrix{
A \times B \ar[d]_{1_A \times \lam} \ar@/^1pc/[dr]^{\gamma} \\
A \times \ogpd(A,C) \ar[r]_-{\ep} & C}
\]
commutes.
\end{lemma}

\begin{proof}
For $b \in B$ with $b \dom =p$ and $b \ran = q$, we define $p\lam$ to be the ordered morphism $A \ra C$ given by
$a(p \lam) = (a,p)\gamma$, and $b \lam$ is the ordered natural transformation $p \lam \natl q \lam$ given by
$x(b \lam) = (x,b)\gamma$ for all $x \in \obj(A)$.  Hence if $a \dom = x$ and $a \ran =y$ we get a commutative square
\[ \xymatrixcolsep{3pc}
\xymatrix{
\cdot \ar[d]_{(a,p)\gamma} \ar[r]^{(x,b)\gamma} & \cdot \ar[d]^{(a,q)\gamma}\\
\cdot \ar[r]_{(y,b)\gamma} & \cdot}
\]
in $C$.  Then
\begin{align*}
(a,b)(1_A \times \lam)\ep & = (a,(p \lam,b \lam))\ep \\
&= a(p\lam)y(b\lam)\\
&= (a,p)\gamma(y,b)\gamma = (a,b)\gamma.
\end{align*}
\end{proof}

The mapping $\nu: \gamma \mapsto \lam$ defined in the lemma defines a function 
\[\nu: \homogpd(A \times B,C) \ra \homogpd(B,\ogpd(A,C))\,.\]
Now given any $\eta : B \ra \ogpd(A,C)$ we can compose $1_A \times \eta : A \times B \ra A \times \ogpd(A,C)$
with $\ep$ to obtain $\delta : A \times B \ra C$:
\[ \xymatrixcolsep{5pc}
\xymatrix{
A \times B \ar[d]_{1_A \times \eta} \ar@/^1pc/[dr]^{\delta} \\
A \times \ogpd(A,C) \ar[r]_-{\ep} & C \,.}
\]
and the mapping $\eta \mapsto \delta$ is inverse to $\nu$.  Hence we have a natural bijection
\[ \nu: \homogpd(A \times B,C) \ra \homogpd(B,\ogpd(A,C)) \,.\]

\begin{cor}
\label{adj_isom}
The bijection $\nu$ extends to a natural isomorphism of ordered groupoids
\[ \nu: \ogpd(A \times B,C) \ra \ogpd(B,\ogpd(A,C)) \,.\]
\end{cor}

\subsection{The endomorphism groupoid}
The ordered functor
\[ A \times \ogpd(A,B) \times \ogpd(B,C) \stackrel{\ep \times 1}{\longrightarrow}
B \times \ogpd(B,C) \stackrel{\ep}{\ra} C \]
corresponds, under the isomorphism of Proposition \ref{adj_isom}, to an ordered functor
\[\mu:  \ogpd(A,B) \times \ogpd(B,C) \ra \ogpd(A,C) \]
called {\em composition}.  On objects, this is just the composition of ordered functors: if $f: A \ra B$ and
$g: B \ra C$ then $(f,g)\mu = fg$.  Now given arrows $(f,\tau)$ and $(g,\sigma)$ in $\ogpd(A,B)$ and $\ogpd(B,C)$
respectively, their composition $((f,\tau),(g,\sigma))\mu = (fg,\phi)$ where, for $x \in \obj(A)$, we have
$x \phi = (x\tau)g((x\tau)\ran)\sigma$.

Of particular interest is the case when $A=B=C$.  We then denote $\ogpd(A,A)$ by $\endo(A)$: the functor
$\mu : \endo(A) \times \endo(A) \ra \endo(A)$ then makes $\endo(A)$ into a monoid in the category of groupoids.
In detail, we have
\[ \endo(A) = \{ (f,\tau): f \in \homogpd(A,A), \tau : \obj(A) \ra A, (x \tau)\dom = xf \} \,.\]
with the monoid operation given by $(f,\tau) \diamond (g,\sigma) = (fg,\tau g \ast \sigma)$, where for $x \in \obj(A)$,
$x(\tau g \ast \sigma) = (x \tau)g ((x \tau)\ran)\sigma$.  

The fact that this is a monoid in the category of groupoids implies that for any four arrows
$(f,\tau),(g,\sigma),(h,\psi),(k,\phi) \in \endo(A)$ with $(f,\tau)(g,\sigma)$ and $(h,\psi)(k,\phi)$ defined in
the groupoid composition on $\endo(A)$, we have the \emph{interchange law}:
\begin{equation}
\label{interchange}
 ((f,\tau)(g,\sigma)) \diamond ((h,\psi)(k,\phi)) = ((f,\tau) \diamond (h,\psi))((g,\sigma) \diamond (k,\phi)) \,.
 \end{equation}
It is worth seeing why this works in the current setting.  
On the left-hand side we have
\begin{align*}
((f,\tau)(g,\sigma)) \diamond ((h,\psi)(k,\phi)) &= (f,\tau \cdot \sigma) \diamond (h,\psi \cdot \phi) \\
&= (fh, (\tau \cdot \sigma)h \ast (\psi \cdot \phi))\\
&= (fh, (\tau h \cdot \sigma h) \ast (\psi \cdot \phi)).
\end{align*}
On the right-hand side we have
\begin{align*}
((f,\tau) \diamond (h,\psi))((g,\sigma) \diamond (k,\phi)) &= (fh,\tau h \ast \psi)(gk ,\sigma k \ast \phi) \\
&= (fh, (\tau h \ast \psi) \cdot (\sigma k \ast \phi)).
\end{align*}
Since $\tau: f \natl g$ and $\psi: h \natl k$, it  is easy to see that $\tau h \ast \psi : fh \natl gk$ and the composition
here is defined.  Now for $x \in \obj(A)$,
\[ x(\tau h \cdot \sigma h) \ast (\psi \cdot \phi) = (x \tau h)(x \sigma h)(x \sigma \ran) \psi (x \sigma \ran) \phi \]
whilst
\[ x(\tau h \ast \psi) \cdot (\sigma k \ast \phi) = (x \tau h)(x \tau \ran)\psi(x \sigma k)(x \sigma \ran)\phi.\]
Because $\psi$ is a natural transformation $h \natl k$ we have the following commutative square for the arrows 
$x \sigma h$ and $x \sigma k$:
\[ \xymatrixcolsep{3pc}
\xymatrix{
\cdot \ar[d]_{xg\psi} \ar[r]^{x \sigma h} & \cdot \ar[d]^{(x \sigma \ran)\psi}\\
\cdot \ar[r]_{x \sigma k} & \cdot}
\]
But here $xgh = (x \tau h)\ran = (x \tau \ran)h$, and so $xg \psi = (x \tau \ran) \psi$, so that
\[ (x \sigma h)(x \sigma \ran) \psi = (x \tau \ran)\psi(x \sigma k) \]
and the interchange law \eqref{interchange} does hold.

The projection $(f,\tau) \mapsto f$ is a monoid homomorphism $\endo(A) \ra \homogpd(A,A)$, and is split by the map
$f \mapsto (f,f|_{\obj(A)})$.

Mappings  $\tau : \obj(A) \ra A$ satisfying $(x \tau)\dom = x$ were studied in \cite{ndgflow} and called \emph{flows} on $A$: the idea of studying flows on a 
category originates, however,  in \cite{chase}.   Flows are called  \emph{arrow fields} in \cite{ish} and its sequels, where they are used in a category-theoretic approach to
quantisation.  The set of all flows on $A$ is a monoid $\Phi(A)$: the composition of two flows $\tau, \sigma$ is the
flow $\tau \ast \sigma : x \mapsto (x \tau)((x \tau)\ran)\sigma$.  The structure of the flow monoid of a connected groupoid $A$ is easy to describe using the
isomorphism between $A$ and $E(A) \times L \times E(A)$ for any local subgroup $L$ of $A$.  We first recall that
for a set $X$ and a group $L$,  the {\em wreath product} $L  \wr  \T(X)$ of $L$ with the
full transformation semigroup $\T(X)$ is a semigroup defined as follows.  The underlying
set is $\{ (\lambda,\theta) : \theta \in \T(X), \lambda : X \ra L \}$
and the semigroup operation is
$(\lambda_1,\theta_1)(\lambda_2,\theta_2)
= (\lambda, \theta_1 \theta_2)$ where 
$x \lambda = (x \lambda_1)((x \theta_1) \lambda_2))$.  
Then \cite[Theorem 6.16]{chase} essentially etablishes the following result.

\begin{theorem} 
\label{flowmonoid}
Let $A$ be a groupoid.
\begin{enumerate}
\item If $A$ is the union of connected components $A_i \,, i \in I$, then
$\Phi(A)$ is isomorphic to the direct product 
$\prod_{i \in I} \Phi(A_i)$.
\item The flow semigroup of a connected groupoid $A$ with local subgroup $L$ is 
isomorphic to the wreath product $L \wr \T(\obj(A))$.
\end{enumerate}
\end{theorem}

For the identity functor $\id_A$, we have $(\id_A,\tau) \in \endo(A)$ if and only if
$\tau : \obj(A) \ra A$ is an ordered mapping satisfying $(x \tau)\dom = x$.  
Therefore $\tau$ is an ordered flow on $A$, and the set $\Phi_{\leq}(A)$ of ordered flows on $A$ is a submonoid of $\Phi(A)$.
The the map $\phi \mapsto (\id_A,\phi)$ embeds $\Phi_{\leq}(A)$ as a submonoid of $\endo(A)$, but the structure of $\Phi_{\leq}(A)$
does not seem to be apparent from Theorem \ref{flowmonoid}, since $E(A) \times L \times E(A)$ does not carry the product ordering.

\section{The holomorph}
Generalising the construction of the holomorph $\hol(G)$ of a group $G$, it is now natural to define the \emph{holomorph} $\hol(S)$ of an inverse semigroup $S$ to be the
ordered groupoid $\endo(\vec{S})$.  Identifying $\homogpd(\vec{S},\vec{S})$ as $\prem(S)$ by Proposition \ref{premeq}, we obtain:
\[ \hol(S) = \{ (\alpha,\tau) : \alpha \in \prem(S), \tau:E(S) \ora S \; \text{ with} \; (e \tau)(e \tau)^{-1} = e \alpha \}.\]
We summarise the outcome of the constructions from section \ref{ogpdiscc}.

\begin{theorem}
\label{holops}
\begin{enumerate}[(a)]
\item For an inverse semigroup $S$, its holomorph $\hol(S)$ is a monoidal groupoid in the cartesian closed category of ordered groupoids.  
\item The groupoid composition of
$(\alpha,\tau)$ and $(\beta,\sigma)$ is given by
\[ (\alpha,\tau)(\beta,\sigma) = (\alpha,\psi) \,,\]
defined when, for all $s \in S$, $s \beta = (ss^{-1})\tau^{-1}(s \alpha) (s^{-1}s)\tau$, and where for all $e \in E(S)$, $e \psi = e \tau e \sigma$.  
\item The monoid
composition of $(\alpha,\tau)$ and $(\beta,\sigma)$ is given by 
 \[ (\alpha,\tau) \diamond (\beta,\sigma) = (\alpha \beta, \tau \beta + \sigma) \,,\]
where, for all $e \in E(S)$, $e(\tau \beta + \sigma) = (e \tau \beta)((e \tau)^{-1} (e \tau))\sigma$.
\item There is a monoid action of $\hol(S)$ on $S$ defined, for all $s \in S$, by
\[ s \lhd (\alpha, \tau) = s \alpha (s^{-1}s)\tau \,. \]
\end{enumerate}
\end{theorem}

\subsection{The heap operation}
We now consider the heap ternary operation on an inverse semigroup $S$, defined for all $a,b,c \in S$ by
$\< a,b,c \> = ab^{-1}c$.  
Suppose that $\eta : S \ra S$ is an ordered function that preserves $\< \cdots \>$ on $S$, so that for all $a,b,c \in S$ we have
$(ab^{-1}c)\eta = (a \eta)(b \eta)^{-1}(c \eta)$.  Now $ab = a(a^{-1}a)b$ and so $(ab)\eta = (a\eta)((a^{-1}a)\eta)^{-1}(b \eta)$.  We define
a function $\phi: S \ra S$ by $a \phi = a \eta ((a^{-1}a)\eta)^{-1}$.  We note that $\phi$ is ordered.  Take $a,b \in S$ with $a^{-1}a=bb^{-1}$.  Then
\[ (ab)\phi = (ab)\eta ((b^{-1}a^{-1}ab) \eta)^{-1} = (a\eta)((a^{-1}a)\eta)^{-1}(b \eta)((b^{-1}b)\eta)^{-1} =a \phi b \phi.\]
Hence $\phi$ is a premorphism.  Moreover,
\begin{align*}
(a^{-1}a)\phi &= (a \phi)^{-1}(a \phi) = (a^{-1}a)\eta(a \eta)^{-1}a \eta ((a^{-1}a)\eta)^{-1} \\
&= (a^{-1}aa^{-1}a)\eta ((a^{-1}a)\eta)^{-1}\\
&= (a^{-1}a)\eta ((a^{-1}a)\eta)^{-1}
\end{align*}
so that $(\phi,\eta) \in \hol(S)$ and, for all $s \in S$ we have 
\[ s \eta = \< s,s^{-1}s,s^{-1}s \> \eta = s \eta ((s^{-1}s)\eta)^{-1}(s^{-1}s)\eta = s \phi (s^{-1}s)\eta \]
so that $s \eta = s \lhd (\phi,\eta)$.  These considerations establish the following result.
\begin{prop}
\label{schar}
Let $\Sha (S)$ denote the monoid of all ordered functions $S \ra S$ that preserve the heap operation.  Then the mapping $\Sha(S)\ra \hol(S)$ given by
$\eta \mapsto (\phi_\eta,\eta|_{E(S)})$ is an embedding, and for all $s \in S$ we have $a \eta = a \lhd (\phi_{\eta},\eta|_{E(S)})$.
\end{prop}

\subsection{Inverse monoids}
Let $M$ be an inverse monoid with identity element $1_M$.  A premorphism $M \ra M$ need not preserve $1_M$ but we do have
$e \theta \leq 1_M \theta$ for all $e \in E(S)$.  As noted in section \ref{ogpdiscc}, if $(\alpha,\tau) \in \hol(M)$ then 
$\tau:E(M) \ra M$ is determined by $1_M \tau$: if $m = 1_M \tau$ and $e \in E(M)$ then $e \tau = (e \alpha)m$.  We can then
replace $\tau$ with $m$.  The definition of the holomorph of $M$ then becomes
\[ \hol(M) = \{ (\alpha,m) : \alpha \in \prem(M), m \in M, mm^{-1} = 1_M \alpha \}.\]
The groupoid composition is given by $(\alpha,m)(\beta,n) = (\alpha,mn)$, defined when $t \beta = m^{-1}(t \alpha)m$
for all $t \in M$, and the monoid composition is given by 
$(\alpha,m) \diamond (\beta,n) = (\alpha \beta,(m \beta)n)$.
This looks like an example of the semidirect product of monoids (see \cite{Nico}) but involving an action by premorphisms rather than
by endomorphisms.  The associativity of $\diamond$ is guaranteed by the considerations in section \ref{ogpdiscc} but can be  
verified directly: for $(\alpha,m), (\beta,n), (\gamma,p) \in \hol(M)$ we have:
\begin{align*}
(\alpha,m) \diamond [(\beta,n)\diamond (\gamma,p)] &= (\alpha \beta \gamma, (m \beta \gamma)(n \gamma)p) \\
\intertext{whereas}
[(\alpha,m) \diamond (\beta,n)] \diamond (\gamma,p) &= (\alpha \beta \gamma, ((m \beta)n) \gamma)p).
\end{align*}
But here $nn^{-1} = 1_M \beta \geq (m^{-1}m)\beta = (m \beta)^{-1}m \beta$ and so by Lemma \ref{prem_is_hom}
$(m \beta \gamma)n \gamma = ((m \beta)n) \gamma$.
The action of $\hol(M)$ on $M$ is given by
$t \lhd (\alpha, m) = (t \alpha)m$. 

\begin{prop}
\label{schar_monoid}
For an inverse monoid $M$, the monoid $\Sha(M)$ is isomorphic to the submonoid $\endomo(M) \ltimes M$ of $\hol(M)$.
\end{prop}

\begin{proof}
By Proposition \ref{schar} we have an embedding $\Sha(M) \ra \hol(M)$ such that the action of $\eta \in \Sha(M)$ on $M$ is
given by the action of the image $(\phi_{\eta},\eta|_{E(S)})$ of $\eta$ in $\hol(M)$.  Now suppose that $(\alpha,m)$
preserves the heap operation, so that for all $a,b,c \in M$ we have
\begin{align*} (ab^{-1}c)\alpha \cdot m &= a \alpha \cdot mm^{-1} (b \alpha)^{-1} c\alpha \cdot m= a \alpha \cdot 1_M \alpha (b \alpha)^{-1} c\alpha
\cdot m \\ &= a\alpha(b \alpha)^{-1} c \alpha \cdot m \,.\end{align*}
Hence 
\begin{align*} (ab^{-1}c)\alpha &= (ab^{-1}c)\alpha 1_M \alpha =  (ab^{-1}c)\alpha mm^{-1} = a\alpha(b \alpha)^{-1} c \alpha \cdot mm^{-1}\\
&= a\alpha(b \alpha)^{-1} c \alpha \cdot 1_M \alpha = a\alpha(b \alpha)^{-1} c \alpha \,, \end{align*}
and so $\alpha \in \Sha(M)$.  But then
\[ (ac)\alpha = \< a, a^{-1}a,c \> \alpha = \< a \alpha, (a^{-1}a)\alpha, b \alpha \> = (a \alpha)((a^{-1}a)\alpha)^{-1}(b \alpha) \,.\]
But $((a^{-1}a)\theta)^{-1} = (a^{-1}a)\theta$ (by Lemma \ref{premorph1}(a)) and then $(a \theta)((a^{-1}a)\theta) = (aa^{-1}a)\theta = a\theta$
by Proposition \ref{premeq}.  $So (ac)\theta = (a \theta)(c \theta)$ and $\alpha \in \endomo(M)$. 
\end{proof}

%
%

\section{Examples}

\subsection{Semilattices of groups}
Let $E$ be a semilattice and $S = (\G,E)$ be a semilattice of groups, with linking maps $\alpha^e_f$ for $e \leq f$ in $E$.  Here $\G$ assigns a group $G_e$ to
each $e \in E$ and $\alpha^e_f$ is a group homomorphism $G_e \ra G_f$.  We have $\alpha^e_e = \id$ for all $e \in E$, and whenever $e \leq f \leq k$ in $E$ then
$\alpha^k_f \alpha^f_e  = \alpha^k_e$.  The product in $S$ of  $g \in G_x$ and $h \in G_y$ is $g\alpha_{xy}^x h\alpha_{xy}^y \in G_{xy}$.  The inductive groupoid
 $\vec{S}$ is a disjoint union of groups
and is ordered by $g \geq g \alpha_f^e$ whenever $e,f \in E$ with $e \geq f$ and $g \in G_e$.

A premorphism of $\vec{S}$ is specified by an order-preserving map $\lambda:E \ra E$ and a family $\phi$ of group homomorphisms $\phi_e : G_e \ra G_{e \lambda}$ such that, if  $e \geq f$ then
$\phi_e \alpha^{e \lambda}_{f \lambda} = \alpha^e_f \phi_f$.  A construction of McAlister \cite[Proposition 4.6]{McA1} shows that premorphisms from $S$ to an inverse semigroup $T$ are in one-to-one
correspondence with idempotent separating homomorphisms from $S$ to a certain semilattice of groups $(\K,E)$ constructed from $T$.  Suppose that $(\lambda,\phi)$ specifies a premorphism $\vec{S} \ra \vec{S}$.
Then $(\K,E)$ obtained as follows: $K_e = G_{e \lambda}$ and the linking map $\beta^e_f$ is equal to $\alpha^{e \lambda}_{f\lambda}$.  Then we obtain
an idempotent separating homomorphism $\sigma : S \ra U$ as follows: we set $\sigma_e = \phi_e  : G_e \ra K_e$, which is clearly idempotent separating.  Then if $g \in G_x$ and 
$h \in G_y$,
\begin{align*}
(gh) \sigma &= ((g \alpha^x_{xy})(h \alpha^y_{xy}))\phi_{xy}\\
&= (g \alpha^x_{xy})\phi_{xy} (h \alpha^y_{xy})\phi_{xy} \\
&= (g \phi_x \alpha^{x \lambda}_{(xy)\lambda})(h \phi_y \alpha^{y \lambda}_{(xy)\lambda}) \\
& = (g \phi_x \beta^x_{xy})(h \phi_y \beta^y_{xy}) \\
&= (g \sigma)(h \sigma).
\end{align*}

Endomorphisms of $S$ are specified by pairs $(\lambda,\phi)$ in which $\lambda$ is meet-preserving.

We have
\[ \hol(\G,E) = \{ (\lambda,\phi,\tau) : (\lambda,\phi) \in \prem(\G,E), \tau: E \ra (\G,E), e\tau \in G_{e \lambda} \}. \]
Since $\tau$ is ordered, it determines a compatible family of elements of the groups $G_{e \lambda}$, in the sense that if $e \geq f$ in $E$ then
$f \tau = (e \tau) \alpha^{e \lambda}_{f \lambda}$.  

Let $(\lambda,\phi,\tau) \in \hol(\G,E)$ and take $a \in G_x, b \in G_y$ and $c \in G_z$.  Then $\<a,b,c \> \lhd (\lambda,\phi,\tau) \in G_{(xyz)\lambda}$
whereas $\< a\lhd (\lambda,\phi,\tau),b\lhd (\lambda,\phi,\tau),c\lhd (\lambda,\phi,\tau)\> \in G_{(x \lambda)(y \lambda)(z \lambda)}$.  Hence
if $(\lambda,\phi,\tau) \in \Sha(\G,E)$  then $\lambda$ is meet-preserving and  $(\lambda,\phi) \in \endomo(\G,E)$, and so the action of $ (\lambda,\phi)$ preserves $\< \cdots \>$.
But then 
\begin{align*}
\< a\rhd & (\lambda,\phi,\tau),b  \lhd (\lambda,\phi,\tau),c\lhd (\lambda,\phi,\tau)\> = \\
&  (a \phi_x (x \tau)) \alpha^{x \lambda}_{(xyz)\lambda}
 (b \phi_y (y \tau))^{-1} \alpha^{y \lambda}_{(xyz)\lambda}
 (c \phi_z (z \tau)) \alpha^{z \lambda}_{(xyz)\lambda} \\
&= (a \alpha^x_{xyx} \phi_{xyz} ) (x \tau \alpha^{x \lambda}_{(xyz)\lambda})
 (y \tau \alpha^{y \lambda}_{(xyz)\lambda})^{-1}  (b \alpha^y_{xyx} \phi_{xyz} )^{-1}
( c \alpha^z_{xyx} \phi_{xyz} ) (z \tau \alpha^{z \lambda}_{(xyz)\lambda}) \\
&= (a \alpha^x_{xyx} \phi_{xyz} )  (b \alpha^y_{xyx} \phi_{xyz} )^{-1}
 (c \alpha^z_{xyx} \phi_{xyz} ) (z \tau \alpha^{z \lambda}_{(xyz)\lambda})\\
\intertext{since $x \tau \alpha^{x \lambda}_{(xyz)\lambda} = (xyz) \tau =  y \tau \alpha^{y \lambda}_{(xyz)\lambda}$}
&= (a \alpha^x_{xyz})( b^{-1} \alpha^y_{xyz})( c \alpha^z_{xyz}) \phi_{xyz} (xyz)\tau\\
&= \< a,b,c \> \lhd (\lambda,\phi,\tau).
\end{align*}
Therefore we have
\[ \Sha(\G,E) = \{ (\lambda,\phi,\tau) : (\lambda,\phi) \in \endomo(\G,E), \tau: E \ra (\G,E), e\tau \in G_{e \lambda} \}. \]

\subsection{The bicyclic monoid}
The bicyclic monoid $B$ is the inverse monoid presented by $\< a:aa^{-1}=1 \>$.  Its idempotents are the elements of the form $a^{-n}a^n,
n \geq 0$, and as a semilattice is an infinite descending chain.  By Proposition \ref{prem_is_hom} every premorphism of $B$  is an endomorphism,
and the endomorphisms of $B$ were described in \cite{Wa}.
Each endomorphism $\nu$ of $B$  is determined by the image of $a$. If $a\nu = a^{-p}a^q$ then $1 \nu = (aa^{-1})\nu=a^{-p}a^p$ and 
$(a^{-1}a)\nu = a^{-q}a^q$.  Since $1 \geq a^{-1}a$ we must have $p \leq q$.  It then follows that
\[ (a^{-i}a^j)\nu = (a^{-p}a^q)^{-i}(a^{-p}a^q)^j=a^{-i(q-p)-p}a^{j(q-p)+p} = a^{-ik-p}a^{jk+p} \]
where $k=q-p$, and so $k \geq 0$, $p \geq 0$.  Hence $\endomo(B)$ is isomorphic to the monoid $\aff(\N)$ of affine transformations of $\N$.  If $(\nu,a^{-l}a^m) \in \hol(B)$
then $a^{-l}a^l = 1 \nu = a^{-p}a^p$, so that $l=p$ but $m \in \N$ is arbitrary.  It follows that
\[ \hol(B) \cong \aff(\N) \ltimes \N \,.\]

\subsection{Polycyclic monoids}
The polycyclic monoids $P_n$ for $n \geq 2$ were introduced by Nivat and Perrot in \cite{NP}. Set $A = \{ a_1, \dotsc , a_n \}$.
Then $P_n$ is the inverse semigroup
with zero presented by
\[ \< A : a_ia_i^{-1}=1, a_ia_j^{-1}=0 \; (i \ne j) \> \,, \]
and its non-zero elements are uniquely representable in the form $u^{-1}v$ for $u,v \in A^*$.
We shall generalise the description of $\hol(B)$ above by computing $\prem(P_n)$.  

Affine transformations of $\N$ generalise to affine maps of a monoid $M$.
An \emph{affine} map
on $M$ is the composition of an endomorphism of $M$ and a right translation: if $\alpha$ is affine then
there exists an endomorphism $\sigma$ and an element $m \in M$ such that, for all $x \in M$,
$\alpha: x \mapsto (x \sigma)m$.  The set of affine maps $\aff(M)$ is then a monoid, and if $M$ is right cancellative it is isomorphic to the semidirect product $\endomo(M) \ltimes M$ of 
$\endomo(M)$ and $M$ (with the natural action of $\endomo(M)$ on $M$).

The ordered groupoid $\vec{P_n}$ can be identified as $\vec{P_n} = \Delta A^* \cup \{ 0 \}$,
where $\Delta A^*$ is the \emph{simplicial groupoid} $A^* \times A^*$ in which a composition $(p,q)(u,v)$ is defined
if and only if $q=u$, and then $(p,q)(q,v)=(p,v)$.  Identity arrows in $\Delta A^*$ (corresponding to non-zero idempotents
in $P_n$) have the form $(u,u), u \in A^*$ and so we identify $E(P_n)$ as $A^* \cup \{ 0 \}$.  The ordering on $A^*$ is the
suffix ordering: 
\[ w \leq u \; \text{if and only if} \; w=pu \; \text{for some} \; p \in A^* \]
with, of course, $0 \leq u$ for all $u \in A^*$. The ordering on $\Delta A^*$ is then
$(pu,pv) \leq (u,v)$ for all $p,u,v \in A^*$.   Since simplicial groupoids are free \cite{pjhbook}, a functor
$\vec{P_n} \ra \vec{P_n}$ is determined by a mapping $E(P_n) \ra E(P_n)$.  

If an ordered functor $\phi$ maps some $u \in A^*$ to
$0$ then it must map the connected component $\Delta A^*$ to $0$.  So we may assume that $\phi :  A^* \ra A^*$.
For $\phi$ to be an ordered mapping it must
be suffix-preserving on $A^*$: that is, if $w=pu$ then $w \phi = q(u \phi)$ for some $q \in A^*$,  where $q$ is uniquely
determined by $p,u$ and $\phi$. The assignment $p \mapsto q$ gives another function $u \rhd \phi: A^* \ra A^*$ that is
also suffix-preserving. 

Let $\S$ be the monoid of all suffix-preserving maps $A^* \ra A^*$.  Then we have a map $X^* \times \S \ra \S$,
$(u,\phi) \mapsto u \rhd \phi$, and the natural (right) action of $\S$ on $A^*$ giving a map $A^* \times \S \ra A^*$.

\begin{prop}
\label{zs_pr}
For all $u,v \in A^+$ and $\phi,\psi \in \S$ we have:
\begin{itemize}
\item $(uv) \rhd \phi = u \rhd (v \rhd \phi)$,
\item $u \rhd \phi \psi = (u \rhd \phi)(u \phi \rhd \psi)$,
\item $u(fg) = (uf)g$,
\item $(uv)\phi = u(v \rhd \phi)(v \psi)$.
\end{itemize}
It follows that the set $\S \times A^*$ is a semigroup with composition 
\[ (\phi,u)(\psi,v) = (\phi(u \rhd \psi),(u \psi)v) \,. \]
\end{prop}
We denote the semigroup in Proposition \ref{zs_pr} by $\S \bowtie A^*$.  It is an example of a \emph{Zappa product} of
semigroups  \cite{Kun}.  Now $X^*$ embeds in $\S$ as the submonoid of right-multiplication maps: $w \mapsto \rho_w \in \S$
where $u \rho_w=uw$, and for all $v,w \in A^*$ we have $v \rhd \rho_w = 1_{\S}$.

\begin{lemma}
\label{map_from_zs}
The mapping $\mu: \S \bowtie A^* \ra \S$ given by $(\phi,u) \mapsto \phi \rho_u$ is a semigroup homomorphism.
\end{lemma}

Now the functor $\vec{P_n} \ra \vec{P_n}$ determined by $\phi$ maps $(u,v) \in \Delta A^*$ to $(u \phi,v \phi)$.
If this is an ordered functor, then for all $p,u,v \in A^*$, 
\[ ((pu)\phi,(pv)\phi) = (q(u \phi),q(v \phi)) \; \text{for some} \; q \in A^* \,.\]
But $q=p(u \rhd \phi)=p(v \rhd \phi)$ and so, for all $u,v \in A^*$ we have
$u \rhd \phi = v \rhd \phi$.  In particular, all the maps $w \rhd \phi$ are equal to $1 \rhd \phi$, where
$1 \rhd \phi$ maps $p \in A^*$ to the prefix to $1\phi$ in $p\phi$:
\[ p \phi = p(1 \rhd \phi)(1 \phi)  \]
and so $w \rhd (1 \rhd \phi) = w \rhd \phi = 1 \rhd \phi$. It follows that $1 \rhd \phi$ is an endomorphism of
$A^*$:
\[ (uv)(1 \rhd \phi) = u(v \rhd \phi) \cdot v(1 \rhd \phi) = u(1 \rhd \phi) \cdot v(1 \rhd \phi) .\]

Suppose that $\phi: 0 \mapsto w \in A^*$.  Then for all $u \in A^*$ we have $w \leq u\phi \leq 1 \phi$.
Suppose that for some $v \in A^*$ we have $v \phi \ne 1 \phi$, and so $v(1 \rhd \phi) \ne 1$. But then, for $m \geq 1$,
\[ (v^m)\phi = (vv^{m-1})\phi = v(v^{m-1} \rhd \phi) (v^{m-1})\phi = v(1 \rhd \phi)(v^{m-1})\phi.\]
Hence the sequence of lengths $(|(v^m)\phi|)$ is strictly increasing, but also bounded below by $|w|$.  This is a 
contradiction, and so for all $v \in A^*$ we have $v \phi = 1 \phi$.

Therefore, there are three types of ordered functors $\vec{P_n} \ra \vec{P_n}$ determined by the following three types of ordered
function $\phi : A^* \cup \{ 0 \} \ra A^* \cup \{ 0 \}$:
\begin{itemize}
\item the constant function $c_0$ with value $0$,
\item functions $c_{w,t}$ that map $0 \mapsto w \in A^*$ and with $u c_{w,t} = t$ for all $u \in A^*$ and some fixed $t \in A^*$ with
$w \leq t$,
\item functions $\phi$ that map $0 \mapsto 0$ and map $A^* \stackrel{\phi}{\ra} A^*$. In this case  $u \phi = u(1 \rhd \phi)(1 \phi)$
with $w \rhd \phi = 1 \rhd \phi$ an endomorphism of $A^*$, and so $\phi = (1 \rhd \phi)\rho_{1 \phi}$.  Hence we have the monoid
\[ \aff(A^*) = \{ \sigma \rho_u : \sigma \in \endomo(A^*), u \in A^* \} \subset \S ,\]
and the restriction of the map $\mu$ in Lemma \ref{map_from_zs} to the semidirect product $\endomo(A^*) \ltimes A^* \subset \S \bowtie A^*$
is an isomorphism.
\end{itemize}

The functor determined by $c_0$ is constant at $0$ and acts as a zero in $\S$, and the composition rule for the $c_{v,t}$
is $c_{u,s}c_{v,t}=c_{t,t}$.  Hence the mappings $c_{w,t},c_0$ determine a subsemigroup $\C$ of $\prem(P_n)$, and $\C$ is an ideal
of $\prem(P_n)$ since
$c_{v,t}\sigma \rho_w = c_{(v\sigma)w,(t \sigma)w} \; \text{and} \; \sigma \rho_w c_{v,t} = c_{v,t}$.

\begin{prop}
\label{prem_polycyclic}
The monoid $\prem(P_n)$ of ordered premorphisms of the polycyclic monoid $P_n$ is an ideal extension of the subsemigroup $\C$ by the
monoid of affine maps $\aff(A^*)$.
\end{prop}

Elements of $\hol(P_n)$ are also of one of three types, determined by the types of elements of $\prem(S)$.  These types are:
\begin{itemize}
\item the element $(c_0,0)$,
\item elements $(c_{w,s},(s,t))$ for $w,s,t \in A^*$ with $s$ a suffix of $w$,
\item elements $(\sigma \rho_u, (u,v))$ with $\sigma \in \endomo(A^*)$ and $u,v  \in A^*$.
\end{itemize}
Certainly $(c_0,0)$ preserves the heap ternary operation, and by Proposition \ref{schar_monoid},

Now for all $(u,v) \in A^*$ we have $(u,v) \lhd  (c_{w,s},(s,t)) = (s,s)(s,t) = (s,t)$ and so
\begin{align*}
 \< (u_1,v_1)  \lhd  (c_{w,s},(s,t)), (u_2,v_2) &  \lhd  (c_{w,s},(s,t)), (u_3,v_3)  \lhd  (c_{w,s},(s,t)) \> \\
&= \< (s,t),(s,t),(s,t) \> = (s,t) .\end{align*}
Hence non-zero values of $\< \dots \>$ are preserved by $(c_{w,s},(s,t))$, but since $0 \lhd (c_{w,s},(s,t)) = (w,w)$ then
instances of $\< \dots \>$ evaluating to $0$ are preserved by $(c_{w,s},(s,t))$ if and only if $w=s=t$.  However,
$(c_{w,w},(w,w))$ acts on $P_n$ in the same way as $(c_1,(w,w)) \in \endomo(P_n) \ltimes P_n$, where $c_1$ is constant at $1$.

We may also determine which premorphisms of $P_n$ are endomorphisms.

\begin{prop}
\label{endo_poly}
Let $\alpha$ be an endomorphism of $P_n$, $(n \geq 2)$. Then either:
\begin{itemize}
\item $\alpha$ is a constant map $c_w$ to some idempotent $w \in E(P_n)$, or:
\item $\alpha: 0 \mapsto 0$ and $(u,v) \mapsto (u \phi,v \phi)$, where $\phi = \sigma \rho_w : A^* \ra A^*$ and $\sigma$ is an injective
endomorphism $A^* \ra A^*$ such that $A \sigma$ is a suffix code.
\end{itemize} 
\end{prop}

\begin{proof}
The Ehresmann-Schein-Nambooripad Theorem (see section \ref{is_and_og}) shows that the endomorphisms of an inverse semigroup $S$
are in one-to-one correspondence with the \emph{inductive} functors $\vec{S} \ra \vec{S}$, that is, the ordered functors that preserve the
meet operation on $E(S)$.  Hence endomorphisms of $P_n$ correspond to ordered functors $\vec{P_n} \ra \vec{P_n}$ that preserve the meet
in the suffix order on $A^* \cup \{ 0 \}$:
\begin{enumerate}
\item[(i)] $0 \cdot u = 0$ for all $u \in A^* \cup \{ 0 \}$,
\item[(ii)] $u \cdot v = 0$ if $u,v \in A^*$ are incomparable in the suffix order,
\item[(iii)] $u \cdot v = v$ if $u,v \in A^*$ and $v$ is a suffix of $u$.
\end{enumerate}
Clearly $c_0$ determines the constant endomorphism to $0 \in P_n$.   Now a map $c_{w,s}$ preserves the meet in cases (i) and (iii) but in case (ii)
we require $s \cdot s = w$.  Hence $s=w$, and the corresponding endomorphism of $P_n$ is constant at $w$.

Now $\sigma \rho_w \in \aff(A^*)$ yields an ordered functor mapping $0 \mapsto 0$ and $u \mapsto (u \sigma)w, u \in A^*$.  Again such a map preserves
the meet in cases (i) and (iii) but in case (ii) we require that $(u \sigma)w$ and $(v \sigma)w$ are incomparable in the suffix order if and only if
$u,v$ are incomparable.  Equivalently, we require $u,v$ to be comparable in the suffix order if and only if $u \sigma, v\sigma$ are.  Since $\sigma$
is an endomorphism of $A^*$, if $u,v$ are comparable then so are  $u \sigma, v\sigma$.  For the converse, we call on \cite[Proposition 2.2]{KaTh}
re-stated for the suffix order on $A^*$: that $u \sigma \leq v \sigma$ implying that $u \leq v$ is equivalent to $\sigma$ being injective with
$A \sigma$ a suffix code in $A^*$.
\end{proof}


\begin{thebibliography}{99}
\bibitem{Baer} R. Baer, Zur Einf\"uhrung des Scharbegriffs. J. Reine Angew. Math. 160 (1929) 199-207.
\bibitem{BrHiSi} R. Brown, P.J. Higgins, and R. Sivera, \emph{Non-abelian Algebraic Topology.} Tracts in Mathematics 15, European Math. Soc. (2011).
\bibitem{cert} J. Certaine, The ternary operation $(abc)=ab^{-1}c$ of a group. Bull. Amer. Math. Soc. 49 (1943) 69-77.
\bibitem{chase} S.U. Chase, On representations of small categories and some
constructions in algebra and combinatorics. Preprint (1977) 84pp.
\bibitem{ndgflow} N.D. Gilbert, Flows on regular semigroups. Applied Categorical
Structures 11  (2003) 147-155.
\bibitem{pjhbook} P.J. Higgins, {\em Notes on categories and
groupoids}. Van Nostrand Reinhold Math. Stud. 32 (1971).  Reprinted
electronically at www.tac.mta.co/tac/reprints/articles/7/7tr7.pdf \,.
\bibitem{ish} C.J. Isham, A new approach to quantising space-time, I. Quantising on a general category.  Adv. Theor. Math. Phys. 7 (2003) 331-367.
\bibitem{KaTh} L.Kari and G. Thierrin, Morphisms and associated congruences. In \emph{Proc. 2nd Intl. Conf. Developments in Language Theory,
Magdeburg 1995}, World Scientific (1996) 119-128.
\bibitem{Kun} M. Kunze, Zappa products. Acta Math. Hungar. 41 (1983) 225-239 .
\bibitem{mvlbook} M.V. Lawson, \emph{Inverse Semigroups.} World Scientific (1998).
\bibitem{McA1} D.B. McAlister, $v$--prehomomorphisms on inverse semigroups.  Pacific J. Math. 67 (1976) 215-231.
\bibitem{McA2} D.B. McAlister, Regular semigroups, fundamental semigroups and groups.  J. Austral. Math. Soc. (Ser. A) 29 (1980) 475-503.
\bibitem{Nico} W.R. Nico, On the regularity of semidirect products. J. Algebra 80 (1983) 29-36.
\bibitem{NP} M. Nivat and J.F. Perrot, Une generalisation du monoide bicyclique, C.R. Acad. Sci.
Paris Ser. I Math. A 271 (1970), 824-827.
\bibitem{Ro} J.S. Rose, {\em A Course on Group Theory.}  Cambridge University Press (1978).
\bibitem{Wa} R.J. Warne, Homomorphisms of $d$--simple inverse semigroups with identity.  Pacific J. Math. 14  (1964) 1111-1122.
\end{thebibliography}
\end{document}